\documentclass[12pt,letterpaper]{article}

\usepackage{amssymb,amsmath,amsthm,mathrsfs,array,graphicx}
\usepackage{float}
\usepackage{multirow}
\usepackage[numbers]{natbib}
\usepackage{subfigure}
\usepackage{caption}
\usepackage{wrapfig}
\usepackage{color}
\usepackage{authblk}
\usepackage{fancyhdr}
\usepackage{lipsum}
\usepackage{amsmath}
\usepackage{mathtools}

\usepackage{enumerate}

\usepackage{textcomp}
\usepackage{siunitx}
\usepackage{enumitem}
\usepackage{tikz} 
	\usetikzlibrary{shapes,arrows}
\usepackage{microtype} 
\usepackage{tcolorbox}
\definecolor{mycolor}{rgb}{0.122, 0.435, 0.698}
\usepackage[margin=1in]{geometry}
\usepackage{algorithmic}

\numberwithin{equation}{section}

\newtheorem{theorem}{Theorem}[section]
\newtheorem{lemma}[theorem]{Lemma}
\newtheorem{remark}[theorem]{Remark}
\newtheorem{definition}[theorem]{Definition}
\newtheorem{corollary}[theorem]{Corollary}
\newtheorem{assumption}[theorem]{Assumption}

\usepackage[small]{titlesec}

\usepackage{hyperref}

\begin{document}

\title{Models of Chemotactic System by Einstein's Brownian Motion Method and its Analysis}
\author{Rahnuma Islam}
\author{Akif Ibragimov}
\affil{Department of Mathematics, Texas Tech University} 

\date{\today}

\maketitle


\begin{abstract}\label{abs}
We study the movement of the living organism in a band form towards the presence of chemical substrates based on a system of partial differential evolution equations. We incorporate Einstein's method of Brownian motion to deduce the chemotactic model exhibiting a traveling band. It is the first time that Einstein's method has been used to motivate equations describing the mutual interaction of the chemotactic system. We have shown that in the presence of limited and unlimited substrate, traveling bands are achievable and it has been explained accordingly. We also study the stability of the constant steady states for the system. The linearized system about a constant steady state is obtained under the mixed Dirichlet and Neumann boundary conditions. We are able to find explicit conditions for linear instability. The linear stability is established with respect to the $L^2$-norm, $H^1$-norm, and $L^\infty$-norm under certain conditions.
\end{abstract}


\section{Introduction}\label{introduction}
The celebrated work of Einstein's theory of Brownian motion \cite{Einstein05} offered the existence of discrete molecules that are too small to be seen through a microscope but the resulting motion should be visible through a microscope. In this theory, he argued that agitated particles in suspended water are the result due to collisions with molecules. Hence he constructed a model governing its motion concerning nearby particles.\par

`Chemotaxis' is a biological phenomenon by which organisms change their state of movement either toward or away from a chemical substance. This migration can be seen in cells ranging from bacteria to mammals. In 1966, Adler conducted an experiment with a capillary tube containing diluted bacteria and chemo-attractant \cite{Adler708}. Bacteria sense the higher gradient of chemo-attractants and move in that direction. During this process of movement towards the chemical gradient, in a detailed inspection, the motion created by each cell appears to be erratic. This randomicity arises not only from the chemotactic response but also from the random jumps of cells.  We argued that Einstein's theoretical framework of Brownian motion can describe the chemotactic response and random motion of an organism. \par

Based on Adler's observation of bands of migrating bacteria, there are many models have been formulated in the past few decades. One of the most fascinating and pioneer chemotaxis models was derived by Keller-Segel \cite{KELLER1971235}. Since then there are extensive research and various variants of Keller-Segel model have been constructed. Where Keller-Segel took a macroscopic approach to deduce the chemotaxis model, another way to model the chemotactic movement of mobile species is from a microscopic perspective. In \cite{stevens1997aggregation}, Othmer and Stevens used a continuous time, discrete-space random walk for the motion of bacteria cells on a one-dimensional lattice. Stevens discussed a stochastic many-particle system where interaction between particles is described by the chemotaxis system which can be interpreted by population densities \cite{stevens2000derivation}. Romanczuk, Erdmann, Engel, and Schimansky-Geier described a self-organized motion of bacteria using the concept of active Brownian particles \cite{romanczuk2008beyond}.\par

In this study, we wanted to adopt the microscopic approach and extend Einstein's random walk framework to derive such a model. Our assumptions involve the interactions between bacteria and the movement of bacteria toward substrates. We also discuss and analyze the aggregated mass of bacterial cells.\par

In Section~\ref{Einstein}, we will derive the chemotactic model motivated by Einstein's random walk model. We present exhibitions of traveling bands in two cases: an environment with an unlimited supply of food and a limited supply of food explained in Section~\ref{unlim}. Section~\ref{lstab} describes the linear instability criterion for constant steady-state solutions with homogeneous Dirichlet and Neumann boundary conditions and the linear stability on $\mathbb{L}^2$ and $L^\infty$ using the energy method.


\section{Derivation of the chemotactic system using Einstein's model}\label{Einstein}
This section explains the steps to incorporate Einstein's Brownian motion model to derive a chemotactic system. 

\subsection{Models with consumption or reaction term}
Let  $X=(x,t)$ is an observable point in space $x\in \mathbb{R}$ at time $t\in (0,\infty)$. Consider a space of observation for $t >0$ bounded by two planes $x$ and $x+ d x$ perpendicular to the $x$-axis. To formulate the partial differential equation (PDE) model, the existence of a time interval $\tau$ between the collision of two particles is required. The interval $\tau$ is ``sufficiently small'' compared to the time scale $t$ of observation of the physical process, but not so small that the motions become correlated. Denote $\Delta$ be the distance each particle makes during the time interval $(t, t+\tau)$ and $\varphi _{\tau}(\Delta)$ be the probability density function of non-collision. Suppose, $w(x,t)$ is the number of the particles (such as bacteria, glucose or predator, prey, etc) in the volume $[x,x+dx]$. We define the following basic properties:
\begin{definition}\label{axiom:Delta-e}(Expected value of the length of free jump)
\begin{equation*}
\Delta_{e}=\int \Delta\varphi _{\tau}(\Delta)d\Delta.
\end{equation*}
\end{definition}

\begin{definition}\label{sigma-b}(Standard variance of free jump)
\begin{equation*}
\sigma^2= \int(\Delta-\Delta_{e})^2 \varphi _{\tau}(\Delta) d \Delta.
\end{equation*}
\end{definition}

Then the number of particles found at time $t+ \tau$ between two planes perpendicular to the $x$-axis, with abscissas $x$ and $x+ d x$, is given by
\begin{align}\label{Eins_b_sys}
&\left(w(x, t+\tau)\right)\cdot dx=  \left(\underbrace{\mathbb{E}[w(x+\Delta,t)]}_{I_1} +\underbrace{w\frac{\partial \Delta_e}{ \partial x}}_{I_2}+\underbrace{\frac{1}{\tau}\int_{t}^{t+\tau} f(x,\xi)d\xi}_{I_3}\right)\cdot dx
\end{align}
Here,
$\mathbb{E}[w(x+\Delta, t)] = \int_{-\infty}^{\infty} w(x+\Delta,t) \varphi_{\tau}(\Delta) d\Delta$ 

In the right-hand side of Eq.~\eqref{Eins_b_sys}, the first term, $I_1$, describes the particle distribution due to random walk. The second term, $I_2$, explains the advective flux of particles dependent on the gradient of the expected length. And the last term, $I_3$, represents the birth or death of particles during $[t,t+\tau]$.

$\tau$, $\Delta$, and $\varphi _{\tau}(\Delta)$ can be functions of spatial distance $x$ and the time variable $t$ and of any other physical quantity such as density or the number of particles, etc. In our case, we will assume, for now, $\tau$ to be independent of the concentration of particles $w(x,t)$. And $\varphi _{\tau}(\Delta)$ is fixed with respect to $w(x,t)$.\par 

We add and subtract $w(x, t)$  on the right-hand side of the Eq.~\eqref{Eins_b_sys} and then we compute as follows
\begin{align}\label{add_sub}
&\big(w(x, t+\tau)- w(x, t)\big)\cdot dx \nonumber\\
&= \Bigg(\mathbb{E}[w(x+\Delta,t)] - \mathbb{E}w[(x,t)] +w\cdot\frac{\partial \Delta_e}{ \partial x}+\frac{1}{\tau}\int_{t}^{t+\tau} f(x,\xi)d\xi \Bigg)\cdot dx. 
\end{align}

Assume that $w(x,t)$ is four time  differentiable function on $\mathbb{R}$ and bounded, then $\big(\mathbb{E}[w(x+\Delta, t)]- w(x,t)\big) $ can be well approximated by formulae \cite{skor}
\begin{equation}\label{skor_lemma}
\big(\mathbb{E}[w(x+\Delta, t)]- w(x,t)\big) = \frac{1}{2}\sigma^2 \frac{\partial^2 w(x,t)}{\partial x^2}+\Delta_e \frac{\partial w(x,t)}{\partial x}.
\end{equation}

Using properties of the function $\varphi$ and applying Eq.~\eqref{skor_lemma} on \eqref{add_sub}, we get 
\begin{align}
\tau \frac{\partial w}{\partial t} =& \Delta_{e}\frac{\partial w}{\partial x}+w\cdot\frac{\partial  \Delta_{e}}{\partial x}+\frac{1}{2}\sigma^2 \frac{\partial^2 w}{\partial x^2}+\frac{1}{\tau}\int_{t}^{t+\tau} f(x,\xi)d\xi,
\end{align}
or, equivalently,
\begin{align}\label{ein-pde}
\tau \frac{\partial w}{\partial t} =& \frac{\partial \left(  w\cdot\Delta_{e}\right)}{\partial x}+ \frac{1}{2}\sigma^2\frac{\partial^2 w}{\partial x^2}+\frac{1}{\tau}\int_{t}^{t+\tau} f(x,\xi)d\xi.
\end{align}


\subsection{Models for chemotaxis systems}
Suppose $x$ is the distance along the tube and $t$ is the time. Denote $u(x,t)$ and $v(x,t)$ to be the concentration of bacteria and chemical substrate (food or any chemical attractor) per unit volume respectively.

\begin{figure}
    \centering
    \includegraphics[width=1\textwidth]{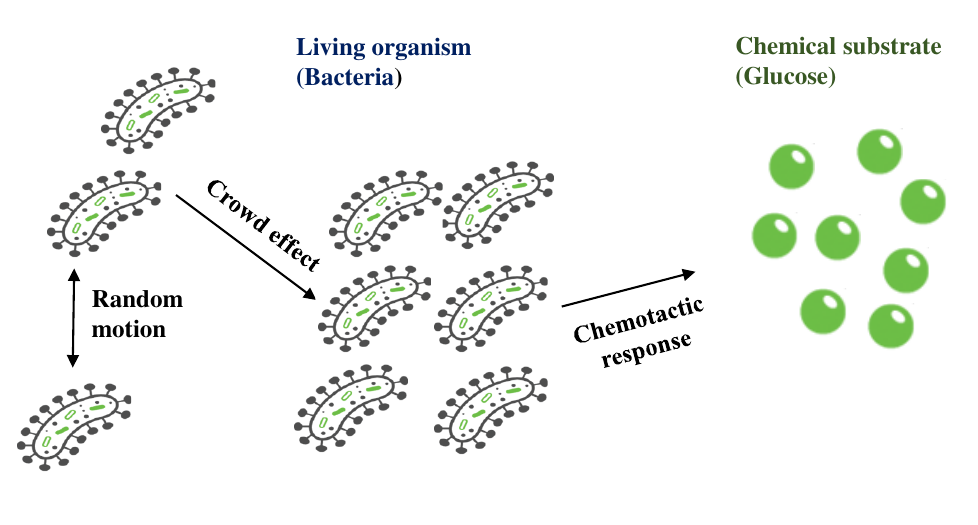}
    \caption{A virtual representation of the interactions between the organism (bacteria) and chemical substrates (Glucose).}
    \label{fig:bacteria_food}
\end{figure}

The corresponding expression of the Eq.~\eqref{ein-pde} for bacteria is
\begin{equation}\label{c_final ein der}
    \tau_u \frac{\partial u}{\partial t}- \frac{\partial \left(  u \cdot\Delta_{e,u}\right)}{\partial x}- \frac{1}{2}\sigma^2_{u}\frac{\partial^2 u}{\partial x^2}-\frac{1}{\tau_u}\int_{t}^{t+\tau_u} f_u(x,\xi)d\xi=0.
\end{equation}

The chemotactic response of the bacteria $u(x,t)$ in the medium can be influenced by the concentration of chemical substrate $v(x,t)$ and/or $\nabla v$. See the recent review \cite{chem}. We assume that the chemotactic response, which causes the event of movement of the organism towards food (or any attractor), is proportional to relative changes of $v$ in space with respect to the amount of food. Then by Definitions~\ref{axiom:Delta-e} and \ref{sigma-b}, $\Delta_{e,u}$ and $\sigma_{u}$ also depend on $v(x,t)$. Keller-Segel suggested a density-dependent sensitivity with a singularity at $v=0$ \cite{KELLER1971235}. Following the Keller-Segel assumption, we suggested the dynamics of directed movement characterized by the expected value of free jump  $\Delta_{e,u}$ as follows, 
\begin{equation}\label{Delab-def-A}
 \Delta_{e,u}(v) = -\beta(v) \frac{\partial v}{\partial x} = - \frac{\beta}{v} \frac{\partial v}{\partial x} = -\beta \frac{\partial \ln v}{\partial x}.
 \end{equation}
$\beta$ is a positive chemotactic coefficient having dimension $ [L^2]$.

Although in real life, the standard deviation is a composite parameter depending on $v, \ \nabla v, \ u, \nabla u, \ x, \ t, etc$, in this article, we consider the dynamics of processes with constant standard deviation. Namely,
\begin{equation}\label{sigma-def}
 \sigma^{2}_{u} (v)= \mu,   
\end{equation}
where $\mu$ is the motility parameter or diffusion coefficient of the organism with dimension $[L^2]$. 
Both $\mu$ and $\beta$  can be obtained from analyses of the dynamics of the process, using image processing. 

$f_u$ is the number of organisms that are born or die per unit volume. We will assume that 
 \begin{equation}\label{Fb-def}
\int_{t}^{t+\tau_{u}} f_u(x,\xi) d\xi= \tau_{u} g(u,v).
\end{equation}
Here $g(u,v)$ is the rate of born or death of the organism with dimension $[ \frac{1}{T}]$. Since the growth or reproduction of an organism happens on a large time scale and chemotaxis occurs on a very small time scale, we can ignore the growth term. Then $g(u,v)=0$ in Eq.~\eqref{b_sys} gives

\begin{equation}\label{b_sys}
         \tau_u \frac{\partial u}{\partial t} + \beta \frac{\partial}{\partial x}\left(u \frac{\partial \ln v}{\partial x}\right) -\frac{\mu}{2}\frac{\partial^2 u}{\partial x^2}= 0.
\end{equation}
 
 The first term on the right-hand side of Eq.~\eqref{b_sys} is the chemotactic response of the organism. 
 The second term is the change in the density of an organism due to random motion.  \par
 
And the concentration $v(x,t)$ of chemical substrates can be given by the equation
\begin{equation}\label{der_s_sys}
    \tau_v \frac{\partial v}{\partial t}- \frac{\partial \left( v \cdot\Delta_{e,v}\right)}{\partial x}- \frac{1}{2}\sigma^2_v\frac{\partial^2 v}{\partial x^2}-\frac{1}{\tau_v}\int_{t}^{t+\tau_v} f_v(x,\xi)d\xi=0.
\end{equation}

\begin{assumption} Food (chemical substrate) is considered to be immovable, so no chemical interaction between particles of substrates is possible under our assumption. Hence,
$$\Delta_{e,v}=0 \text{   and   } \sigma^{2}_{v} = D,$$
with $D$ being the diffusion constant of the chemical substrate. 
\end{assumption} 

\begin{assumption} $f_v$ is defined to be the consumption of substrate cells and
\begin{equation*}\label{cons-of-s}
\int_{t}^{t+\tau_{v}} f_v(x,\xi) d\xi= H(u,v)= - \tau_{v}k(v)u,
\end{equation*}
where $k(v)$ is the rate of consumption of the substrate with dimension $[\frac{1}{T}]$. 
\end{assumption}

Under assumptions, Eq.~\eqref{der_s_sys} can be written as
\begin{equation}\label{s_sys}
    \tau_{v} \frac{\partial v}{\partial t}  = -\tau_{v} k(u,v)u+ \frac{D}{2} \frac{\partial^2 v}{\partial x^2}.
\end{equation}

Hence the chemotactic model is
\begin{subequations}\label{main_mod}
\begin{align}
    \tau_u \frac{\partial u}{\partial t} + \beta \frac{\partial}{\partial x}\left(u \frac{\partial \ln v}{\partial x}\right) -\frac{\mu}{2}\frac{\partial^2 u}{\partial x^2}&= 0,\\
    \tau_{v} \frac{\partial v}{\partial t} - \frac{D}{2} \frac{\partial^2 v}{\partial x^2} +\tau_{v} k(u,v)u&=0.
\end{align}
\end{subequations}


\section{Traveling bands}\label{unlim}
This section is dedicated to showing that chemotactic models in the presence of unlimited and limited substrate exhibits traveling band. First, we will define the traveling band. 

\begin{definition}\label{travel_band} (Traveling Band)
The system of Eqs.~\eqref{b_sys} and \eqref{s_sys} exhibits a traveling band  form if there exist solutions $u(x,t)$ and $v(x,t)$ of the following form
\begin{equation}\label{chng var}
    u(x,t)=U(x-ct) \text{   and    } v(x,t)=V(x-ct) \text{   for all    } x \in \mathbb{R}  \text{   and    } t \geq0
\end{equation}
where $c>0$ is the constant band speed, and $U$, $V$ are functions from $\mathbb{R}$ to $(0, \infty)$ such that $\lim_{\zeta \to \pm \infty} U(\zeta)$ and $\lim_{\zeta \to \pm \infty} V(\zeta)$ exist and belong to $[0,\infty)$.
\end{definition}

We will also assume $D=0$ in the Eq.~\eqref{s_sys} since its effect is trivial in the chemotactic model. And, for the sake of simplicity, we will use $\tau=\tau_u$ and $\tau_v =1$. 


\subsection{The case of unlimited substrates
}\label{unlimited_alpha=0}

In the presence of an abundance of the substrate, the rate of consumption of the food, $k(v)$, does not depend on the concentration of food. Therefore $k(v) = k = constant$.
Then the chemotactic model for unlimited availability of substrate is
\begin{subequations}\label{mod1}
 \begin{align} 
     \tau \frac{\partial u}{\partial t} +\beta \frac{\partial}{\partial x}\left(u \frac{\partial \ln v}{\partial x}\right)- \frac{\mu}{2}\frac{\partial^2 u}{\partial x^2} &= 0,\label{red b_sys_eq}\\
     \frac{\partial v}{\partial t} +k u &= 0,\label{red s_sys_eq}
 \end{align}
\end{subequations}
for all $\zeta$ in $\mathbb{R}$.

\begin{theorem}\label{unlim_food}
If $d = \frac{2 \beta}{\mu}$ and $d \geq 1$, then the system \eqref{mod1} exhibits traveling band form. More specifically, for any $\tau, \beta, \mu, k, c, V_{\infty}, C_{0} >0$ and $d\geq1$, there exist solutions of the form \eqref{chng var} where
\begin{subequations}\label{mod1_soln}
\begin{align}
    U(\zeta) &= C_{0} V^{d}(\zeta) e^{-\frac{2\tau c \zeta}{\mu}},\label{soln b_sys_old}\\
  V(\zeta) &= \begin{cases} 
      \big[\frac{1}{2}C_{0} k c^{-2}\tau^{-1} \mu (d-1) e^{-\frac{2\tau c \zeta}{\mu}}+V_{\infty}^{-d+1}\big]^{-\frac{1}{d-1}} & \text{  for  } d>1,\\
      V_{\infty} e^{-\frac{1}{2}C_{0} kc^{-2}\tau^{-1}\mu e^{-\frac{2 \tau c}{\mu}\zeta}} & \text{  for  } d=1.
   \end{cases}\label{soln s_sys}
\end{align}
\end{subequations}
Moreover, $(U(\zeta), V(\zeta))$ is the solution of Eqs.~\eqref{ch red b_sys} and \eqref{ch red s_sys} that satisfies
\begin{align}
    &U(\zeta) \xrightarrow{} 0 \text{,    } U^{'}(\zeta) \xrightarrow{} 0 \text{,     } V(\zeta) \xrightarrow{} V_\infty \text{      as   } \zeta \xrightarrow{} \infty, \label{bd_cond1}\\
    &U(0) V(0)^{-d} = C_{0}.\nonumber
\end{align}
\end{theorem}

\begin{proof}
With $u(x,t)$ and $v(x,t)$ as in Eq.~\eqref{chng var} and Eq.~\eqref{chng var}, the Eqs.~\eqref{red b_sys_eq} and \eqref{red s_sys_eq} are reduced to
\begin{align}
     \tau c U^{'} - \beta \big(U V^{-1} V^{'} \big)^{'}+ \frac{\mu}{2} U^{''} \label{ch red b_sys}&=0,\\
     c V^{'}  - k U\label{ch red s_sys}&=0.
\end{align}
We will look for solutions $U$ and $V$ of Eqs.~\eqref{ch red b_sys} and \eqref{ch red s_sys} that satisfy \eqref{chng var}.\\

Dividing Eq.~\eqref{ch red b_sys} by $\frac{\mu}{2}$ and integrating once we obtain
\begin{equation}\label{first_int_b}
    \frac{2\tau c}{\mu} U - d U V^{-1} V^{'} + U^{'} = constant.
\end{equation}
Then taking the limit $\zeta \to \infty$ in Eq.~\eqref{first_int_b} and applying the conditions \eqref{bd_cond1}, we get constant to be 0. Then Eq.~\eqref{first_int_b} gives
\begin{equation*}\label{(ln u)'-equation}
   U^{'} =  \big(d V^{-1} V^{'} - \frac{2 \tau c}{\mu}\big) U.
\end{equation*}
Therefore, the solution for $U$ is
\begin{equation}\label{mod1_u}
    U(\zeta)= C_{0} V^{d}(\zeta) e^{-\frac{2\tau c \zeta}{\mu}} =  U(0) V(0)^{d} V^{d}(\zeta) e^{-\frac{2\tau c \zeta}{\mu}}.
\end{equation}

\textit{\textbf{Case $d>1$.}} From Eq.~\eqref{ch red s_sys}, we get $V^{'} = \frac{k}{c} U$.

Substituting the expression of $U$ given by Eq.~\eqref{mod1_u} into Eq.~\eqref{ch red s_sys} and integrating with conditions \eqref{bd_cond1}, we get
\begin{equation*}
    V(\zeta)=\big[\frac{1}{2}C_{0} k c^{-2}\tau^{-1} \mu (d-1) e^{-\frac{2\tau c \zeta}{\mu}}+V_{\infty}^{-d+1}\big]^{-\frac{1}{d-1}}.
\end{equation*}

\textit{\textbf{Case $d=1$.}}
The solution of $U$ is
\begin{align*}
    U(\zeta) = C_{0} V(\zeta) e^{-\frac{2\tau c \zeta}{\mu}}.
\end{align*}
Then Eq.~\eqref{ch red s_sys} gives 
\begin{align}
    V^{'} = \frac{k}{c} C_{0}V e^{-\frac{2\tau c \zeta}{\mu}}.\label{d=1}
\end{align}
Integrating the Eq.~\eqref{d=1} and applying the condition \eqref{bd_cond1} we get
\begin{align*}
    V(\zeta) = V_{\infty} e^{-\frac{1}{2}C_{0} kc^{-2}\tau^{-1}\mu e^{-\frac{2 \tau c\zeta}{\mu}}}.
\end{align*}
\end{proof}

One can verify that the solutions \eqref{soln b_sys_old} and \eqref{soln s_sys} satisfy the condition \eqref{bd_cond1} and 
\begin{align*}
    \lim_{\zeta \xrightarrow{} -\infty}U(\zeta) = 0 &\text{     and     } \lim_{\zeta \xrightarrow{} -\infty}V(\zeta) = 0.
\end{align*}

Also, using Eq.~\eqref{ch red s_sys}, we can compute
\begin{align*}
    c = \frac{k}{V_{\infty}} \int_{-\infty}^{\infty} U(\zeta) \, d\zeta.
\end{align*}

\begin{remark}
Here $V_{\infty}>0$ is the certain threshold concentration of substrate that initiates the consumption of substrate by bacteria.
\end{remark}
The closed form of the solution allows basic sensitivity analyses of the model with respect to its parameters as follows. 
\begin{theorem}
Denote $\Vec{W} = (d, C_0, k, c, \tau, \mu,  V_{\infty})\in \mathbb{D} = (1, \infty) \times (0, \infty)^6$.Then for any compact set $\Omega \subset \mathbb{D}$, there exists $K>0$ such that, for any $\Vec{W_1}, \Vec{W_2} \in \Omega$, 
\begin{subequations}\label{lipschitz}
\begin{align}
    |U(\zeta, \Vec{W_1})- U(\zeta, \Vec{W_2})|&\leq K |\Vec{W_1}-\Vec{W_2}|,\label{lip_U}\\
     |V(\zeta, \Vec{W_1})- V(\zeta, \Vec{W_2})|&\leq K |\Vec{W_1}-\Vec{W_2}|\label{lip_V}
\end{align}
\end{subequations}
for all $\zeta \in \mathbb{R}$.
\end{theorem}

\begin{proof}
Let $A (C_{0}, k, c, \tau, \mu, \zeta) = \frac{1}{2}C_{0} kc^{-2}\tau^{-1}\mu e^{-\frac{2\tau c \zeta}{\mu}}$, then Eq.~\eqref{soln s_sys} can be written as
\begin{align*}
  V(\zeta) = \left(A (d-1) +V_{\infty}^{-(d-1)}\right)^{-\frac{1}{(d-1)}}.
\end{align*}

Without loss of generality, assume 
$$\Omega = [d_{*}, d^{*}] \times [C_{*}, C^{*}] \times [k_{*}, k^{*}] \times [c_{*}, c^{*}] \times [\tau_{*}, \tau^{*}] \times [\mu_{*}, \mu^{*}] \times [V_{*}, V^{*}] \subset \mathbb{D}.$$

We prove \eqref{lip_V} first. For $\Vec{W} \in \Omega$, we compute
\begin{align*}
  D_{\Vec{W}}V(\zeta, \Vec{W}) = \frac{\partial V}{\partial \Vec{W}}=\left(\frac{\partial V}{\partial d},\frac{\partial V}{\partial C_{0}}, \frac{\partial V}{\partial k}, \frac{\partial V}{\partial c}, \frac{\partial V}{\partial \tau}, \frac{\partial V}{\partial \mu},\frac{\partial V}{\partial V_{\infty}}\right),
\end{align*}
where
\begin{subequations}\label{der_V}
\begin{align}
    \frac{\partial V}{\partial d} & = -V \frac{1}{d-1} \left(\ln{V} + \left(A (d-1)+ V_{\infty}^{-(d-1)}\right)^{-1} (A - V_{\infty}^{-(d-1)}\ln{V_{\infty}})\right) ,\label{der_d_V}\\
    \frac{\partial V}{\partial C_{0}} & =  - \left(A (d-1) +V_{\infty}^{-(d-1)}\right)^{-\frac{d}{d-1}} \frac{\partial A}{\partial C_{0}},\label{der_C_V}\\ 
    \frac{\partial V}{\partial k} & =  - \left(A (d-1) +V_{\infty}^{-(d-1)}\right)^{-\frac{d}{d-1}} \frac{\partial A}{\partial k},\label{der_k_V}\\
    \frac{\partial V}{\partial c} & =  - \left(A (d-1) +V_{\infty}^{-(d-1)}\right)^{-\frac{d}{d-1}} \frac{\partial A}{\partial c},\label{der_c_V} \\
    \frac{\partial V}{\partial \tau} & =  - \left(A (d-1) +V_{\infty}^{-(d-1)}\right)^{-\frac{d}{d-1}} \frac{\partial A}{\partial \tau},\label{der_tau_V}\\
    \frac{\partial V}{\partial \mu} & =  - \left(A (d-1) +V_{\infty}^{-(d-1)}\right)^{-\frac{d}{d-1}} \frac{\partial A}{\partial \mu},\label{der_mu_V}\\
    \frac{\partial V}{\partial V_{\infty}} & =  \left(A (d-1) +V_{\infty}^{-(d-1)}\right)^{-\frac{d}{d-1}} V_{\infty}^{-d}\label{der_Vinf_V}.
\end{align}
\end{subequations}

Now, we will take a closer look at Eq.~\eqref{der_d_V} as follows and the rest of the equations can be investigated similarly. For any $\zeta \in \mathbb{R}$ and for any $\Vec{W}\in \Omega$, Eq.~\eqref{der_d_V} gives  

\begin{align*}
    \left|\frac{\partial V}{\partial d}(\zeta, \Vec{W})\right| 
    & \leq V^{*}\frac{1}{d_{*}-1}\bigg(\frac{1}{d_{*}-1}\left|\ln{(A(d-1) + V_{\infty}^{-(d-1)})}\right| + \frac{1}{d_{*}-1} \\
    & \quad + \max\{\left|\ln{V^{*}}\right|,\left|\ln{V_{*}}\right|\}\bigg).
\end{align*} 

For $\zeta >0$, we have $0<A \leq \frac{1}{2}C^{*}k^{*}c^{-2}_{*}\tau^{-1}_{*}\mu^{*}$. Then
\begin{align*}
    A(d-1)+V_{\infty}^{-(d-1)} \geq \frac{1}{(V^{*}+1)^{d_*-1}} = k_{1},\\ 
    A(d-1)+V_{\infty}^{-(d-1)} \leq \frac{1}{2}C^* k^*c_{*}^{-2}\tau_{*}^{-1}\mu^{*} + (\frac{1}{V_*}+1)^{d^*-1}=k_2.
\end{align*}

Then we get
\begin{align*}
     \left|\frac{\partial V}{\partial d} (\zeta, \Vec{W})\right|  \leq V^{*}\frac{1}{d_{*}-1}\big(&\frac{1}{d_{*}-1}\max\{|\ln{k_1}|, |\ln{k_2}|\}\\
    & + \frac{1}{d_{*}-1}+ \max\{\left|\ln{V^{*}}\right|,\left|\ln{V_{*}}\right|\}\big) = K_1.
\end{align*} 

Thus,
\begin{align}\label{k1_d}
    \left|\frac{\partial V}{\partial d} (\zeta, \Vec{W})\right| \leq K_1
\end{align}
for any $\zeta \geq 0$ and $\Vec{W} \in \Omega$.

For $\zeta < 0$, we get
\begin{align*}
    &\frac{\partial V}{\partial d} (\zeta, \Vec{W}) = 
     - \frac{e^{\frac{2\tau c \zeta}{\mu}\frac{1}{d-1}}}{(d-1)\big(\frac{1}{2}C_{0} kc^{-2}\tau^{-1}\mu (d-1) +V_{\infty}^{-(d-1)}e^{\frac{2\tau c \zeta}{\mu}}\big)^{\frac{1}{d-1}}}\\
    &\quad \cdot \left(-\frac{2 \tau c \zeta}{\mu}+\ln{\left(\frac{1}{2}C_{0} kc^{-2}\tau^{-1}\mu (d-1) +V_{\infty}^{-(d-1)}e^{\frac{2\tau c \zeta}{\mu}}\right)} + \frac{1}{d-1}+ \left|\ln{V_{\infty}}\right|\right),\\
    & = I_{1}+I_2,
\end{align*}
where
\begin{align*}
    I_1 &= - \frac{e^{\frac{2\tau c \zeta}{\mu}\frac{1}{d-1}}\ln{\big(\frac{1}{2}C_{0} kc^{-2}\tau^{-1}\mu (d-1)}+V_{\infty}^{-(d-1)}e^{\frac{2\tau c \zeta}{\mu}}\big)}{(d-1)\left(\frac{1}{2}C_{0} kc^{-2}\tau^{-1}\mu (d-1) +V_{\infty}^{-(d-1)}e^{\frac{2\tau c \zeta}{\mu}}\right)^{\frac{1}{d-1}}},\\
    I_2 &= - \frac{e^{\frac{2\tau c \zeta}{\mu}\frac{1}{d-1}}}{(d-1)\left(\frac{1}{2}C_{0} kc^{-2}\tau^{-1}\mu (d-1) +V_{\infty}^{-(d-1)}e^{\frac{2\tau c \zeta}{\mu}}\right)^{\frac{1}{d-1}}} \left(-\frac{2 \tau c \zeta}{\mu} + \frac{1}{d-1}+ \left|\ln{V_{\infty}}\right|\right).
\end{align*}

We estimate $|I_1|$ first. We have 
\begin{align*}
    \frac{1}{2}C_{0} kc^{-2}\tau^{-1}\mu (d-1)+V_{\infty}^{-(d-1)}e^{\frac{2\tau c \zeta}{\mu}} 
    &\geq \frac{1}{2}C_{*} k_*c^{*-2}\tau^{*-1}\mu_* (d_*-1)=k_3,\\
    \frac{1}{2}C_{0} kc^{-2}\tau^{-1}\mu (d-1)+V_{\infty}^{-(d-1)}e^{\frac{2\tau c \zeta}{\mu}} 
    &\leq \frac{1}{2}C^{*} k^{*}c_*^{-2}\tau_*^{-1}\mu^* (d^*-1)+(\frac{1}{V_{*}}+1)^{-(d_*-1)}=k_4,\\
    \big(\frac{1}{2}C_{0} kc^{-2}\tau^{-1}\mu (d-1) +V_{\infty}^{-(d-1)}e^{\frac{2\tau c \zeta}{\mu}}\big)^{-\frac{1}{d-1}} &\leq \big(\frac{1}{\frac{1}{2}C_{*} k_{*}c^{*-2}\tau^{*-1}\mu_{*} (d_{*}-1)}+1\big)^{\frac{1}{d_*-1}}= k_5.
\end{align*}

Hence
\begin{align*}
    \left|I_{1}\right| \leq  \frac{k_5\max\{|\ln{k_3}|,|\ln{k_4}|\}}{d_{*}-1} =K_{2}.
\end{align*}

For estimating $|I_2|$, we have
\begin{align*}
    \left|I_{2}\right| &\leq V^{*}\frac{k_5k_6 \frac{2 \tau^{*} c^{*}}{\mu_{*}}}{d_{*}-1}+ V^{*}\frac{k_5\left( \frac{1}{d_{*}-1}+ 
    \max\{\left|\ln{V^{*}}\right|,\left|\ln{V_{*}}\right|\}\right)}{d_{*}-1} = K_3,
\end{align*}
where,
\begin{align}\label{k_6}
    \sup\{e^{-\frac{2\tau_{*} c_{*} |\zeta|}{\mu^{*}(d^*-1)}} |\zeta|, \zeta<0\} = k_6 < \infty. 
\end{align}

Hence
\begin{align}\label{k2_d}
    \left|\frac{\partial V}{\partial d} (\zeta, \Vec{W})\right| \leq K_2+K_3 = K_4
\end{align}
for any $\Vec{W} \in \Omega$ and $\zeta <0$.
Let $K = \max\{K_1, K_4\}$, then by Eqs.\eqref{k1_d} and \eqref{k2_d}, we get
\begin{align}
    \left|\frac{\partial V}{\partial d} (\zeta, \Vec{W})\right| \leq K
\end{align}
for any $\Vec{W} \in \Omega$ and $\zeta \in \mathbb{R}$.

A similar inspection on Eqs.~\eqref{der_V} reveals that, for any $\Vec{W} \in \Omega$, there exists $K>0$ such that $|D_{\Vec{W}}V| \leq K$ on $\mathbb{R} \times \Omega$. Therefore, by Mean Value Theorem we obtain Eq.~\eqref{lip_V}.

Similarly, we can show that for any $\Vec{W} \in \Omega$ and $\zeta \in \mathbb{R}$,
\begin{align*}
  D_{\Vec{W}}U(\zeta, \Vec{W}) = \frac{\partial U}{\partial \Vec{W}}=\big(\frac{\partial U}{\partial d},\frac{\partial U}{\partial C_{0}}, \frac{\partial U}{\partial k}, \frac{\partial U}{\partial c}, \frac{\partial U}{\partial \tau}, \frac{\partial U}{\partial \mu},\frac{\partial U}{\partial V_{\infty}}\big),
\end{align*}
with
\begin{subequations}\label{der}
\begin{align}
    \frac{\partial U}{\partial d} & = C_{0} d V^{d-1} \frac{\partial V}{\partial d} e^{-\frac{2\tau c \zeta}{\mu}},\label{der_U_d}\\
    \frac{\partial U}{\partial C_{0}} & =  V^{d} \big(1+ C_{0} d V^{-1} \frac{\partial V}{\partial C_{0}}\big) e^{-\frac{2\tau c \zeta}{\mu}},\label{der_U_C}\\ 
    \frac{\partial U}{\partial k} & = C_{0} d V^{d-1} \frac{\partial V}{\partial k} e^{-\frac{2\tau c \zeta}{\mu}},\label{der_U_k}\\
    \frac{\partial U}{\partial c} & = C_{0} V^{d}  \big(d V^{-1} \frac{\partial V}{\partial c} - \frac{2 \tau \zeta}{\mu}\big)e^{-\frac{2\tau c \zeta}{\mu}}, \label{der_U_c}\\
    \frac{\partial U}{\partial \tau} & =  C_{0} V^{d}  \big(d V^{-1} \frac{\partial V}{\partial \tau} - \frac{2 c \zeta}{\mu}\big)e^{-\frac{2\tau c \zeta}{\mu}},\label{der_U_tau}\\
    \frac{\partial U}{\partial \mu} & =  -C_{0} V^{d} \big(d \big[A (d-1) +V_{\infty}^{-(d-1)}\big]^{-1} \frac{\partial A}{\partial \mu} - \frac{2 \tau c \zeta}{\mu^2}\big)e^{-\frac{2\tau c \zeta}{\mu}},\label{der_U_mu}\\
    \frac{\partial U}{\partial V_{\infty}} & =  C_{0} d V^{d-1} \frac{\partial V}{\partial V_{\infty}} e^{-\frac{2\tau c \zeta}{\mu}}.\label{der_U_V}
\end{align}
\end{subequations}
are bounded and thus the Eq.~\eqref{lip_U} is satisfied.
\end{proof}

We also investigate the uniform convergence of the solutions~\eqref{soln b_sys_old} and \eqref{soln s_sys} with respect to the parameter $d$. Here, $d$ is the ratio of the chemotactic coefficient and motility coefficient. 
\begin{theorem}
Suppose $C_0, k, c, \tau, \mu, V_{\infty}$ are given and denote the family of the solutions given by the Eqs.~\eqref{soln b_sys_old} and \eqref{soln s_sys} by $U_{d}(\zeta)$ and $V_{d}(\zeta)$ corresponding to the values of $d \geq 1$. Then 
\begin{align*}
    U_{d}(\zeta) \to U_{1}(\zeta) \text{    and    } V_{d}(\zeta) \to V_{1}(\zeta) \text{    uniformly in $\zeta \in \mathbb{R}$ as    } d \searrow 1.
\end{align*}
\end{theorem}

\begin{proof}
Let $\delta = d-1$, $B = \frac{1}{2}C_{0} kc^{-2}\tau^{-1}\mu$ and $y = B V_{\infty}^{\delta} e^{-\frac{2\tau c \zeta}{\mu}}$, then Eq.~\eqref{soln s_sys}, for $d>1$, is reduced to
\begin{align}\label{v_red_ucont}
    V_{1+\delta}(\zeta) = V_{\infty}\big((1+\delta y)^{\frac{1}{\delta y}}\big)^{-B V_{\infty}^{\delta} e^{-\frac{2\tau c \zeta}{\mu}}}
\end{align}
Then 
\begin{equation*}
    \lim_{\delta \to 0} V_{1+\delta}(\zeta) = V_{\infty} e^{-B e^{-\frac{2 \tau c}{\mu}\zeta} }= V_{1}(\zeta).
\end{equation*}
And
\begin{equation*}
    \lim_{\delta \to 0} U_{1+\delta}(\zeta) = C_{0} V_{1} e^{-\frac{2\tau c \zeta}{\mu}}=U_{1}(\zeta).
\end{equation*}

Next, we need to prove it converges uniformly on $\mathbb{R}$. Given $\epsilon>0$, let denote 
\begin{align}\label{V_exp}
    V_{1+\delta}(\zeta) = V_{\infty}\big((1+\delta y)^{\frac{1}{\delta y}}\big)^{-y} = V_{\infty} \big(e^{{\frac{1}{\delta y}} \ln(1+ \delta y)}\big)^{-y} = V_{\infty} e^{-{\frac{1}{\delta}} \ln(1+ \delta y)}.
\end{align}

\medskip
\noindent
\textbf{Step 1.} There exists $M>0$ so that for any $\zeta \in (-\infty, -M]$, we have
\begin{align}\label{step1_V1}
    |V_{1}(\zeta)| \leq |V_{\infty} e^{-B e^{\frac{2 \tau c}{\mu}M}}| < \frac{\epsilon}{3}.
\end{align}
And Eq.~\eqref{v_red_ucont} gives
\begin{align}\label{step1_Vd_norm}
    ||V_{1+ \delta}(\zeta)||_{\mathcal{C}((-\infty, -M])}&= V_{\infty}(1+ \delta B V_{\infty}^{\delta} e^{\frac{2 \tau c}{\mu}M})^{-\frac{1}{\delta}}.
\end{align}
Because 
\begin{equation*}
    \lim_{\delta \to 0} V_{\infty}(1+ \delta B V_{\infty}^{\delta} e^{\frac{2 \tau c}{\mu}M})^{-\frac{1}{\delta}} = V_{\infty}e^{-B e^{\frac{2 \tau c}{\mu}M}},
\end{equation*}
then there exists $\delta_{1}>0$ such that for any $0< \delta < \delta_1$, we have
\begin{align}\label{step1_Vd}
    \left|V_{\infty}(1+ \delta B V_{\infty}^{\delta} e^{\frac{2 \tau c}{\mu}M})^{-\frac{1}{\delta}}-V_{\infty} e^{-B e^{\frac{2 \tau c}{\mu}M}}\right| & <\frac{\epsilon}{3}.
\end{align}
Thus
\begin{align*}
    ||V_{1+ \delta}(\zeta)||_{\mathcal{C}((-\infty, -M])} & < \frac{\epsilon}{3} + V_{\infty} e^{-B e^{\frac{2 \tau c}{\mu}M}} \leq \frac{\epsilon}{3} + \frac{\epsilon}{3} = \frac{2\epsilon}{3}.
\end{align*}
Therefore, by Eqs.~\eqref{step1_V1}, \eqref{step1_Vd_norm} and \eqref{step1_Vd} we get, for $\zeta \in (-\infty, -M]$,
\begin{align}\label{step1}
    |V_{1+\delta}(\zeta) - V_{1}(\zeta)| \leq |V_{1+\delta}(\zeta)| + |V_{1}(\zeta)| \leq \frac{2\epsilon}{3} + \frac{\epsilon}{3} = \epsilon.
\end{align}

\medskip
\noindent
\textbf{Step 2.} Consider $\zeta \in [-M, \infty)$ and $\delta <1$. Then
\begin{align}
    0< y \leq K:= B (V_{\infty} +1) e^{\frac{2 \tau c M}{\mu}}.
\end{align}
In the last formula of Eq.~\eqref{V_exp}, expanding $\ln(1+\delta y)$ in $\delta y \in [0,K]$ gives
\begin{align}\label{expand}
    V_{1+\delta}(\zeta) 
    & = V_{\infty}e^{-\frac{1}{\delta} (\delta y + \mathbb{O}((\delta y)^2))}= V_{\infty}e^{-y + \mathbb{O}(\delta)}.
\end{align}

For any $\zeta \in [-M, \infty)$,
\begin{align*}
    |V_{1+\delta}(\zeta) - V_{\infty} e^{-B e^{-\frac{2 \tau c}{\mu}\zeta}}| & \leq V_{\infty} \left|e^{(1-V_{\infty}^{\delta})Be^{-\frac{2 \tau c}{\mu}\zeta}+ \mathbb{O}(\delta)} - 1\right|.
\end{align*}
Because $e^{-\frac{2 \tau c}{\mu}\zeta} \leq e^{\frac{2 \tau c}{\mu}M}$ for all $\zeta \in [-M, \infty)$, then $(1-V_{\infty}^{\delta})Be^{-\frac{2 \tau c}{\mu}\zeta}+ \mathbb{O}(\delta)$ converges to $0$ as $\delta \to 0$ uniformly in $\zeta \in [-M, \infty)$. Therefore, $e^{(1-V_{\infty}^{\delta})Be^{-\frac{2 \tau c}{\mu}\zeta}+ \mathbb{O}(\delta)}$ converges to 1 as $\delta \to 0$ uniformly in $\zeta$. Thus, there exists $\delta_2>0$ such that for any $0< \delta < \delta_2$ and $\zeta \in [-M, \infty)$,
\begin{align*}
    V_{\infty} \left|e^{(1-V_{\infty}^{\delta})Be^{-\frac{2 \tau c}{\mu}\zeta}+ \mathbb{O}(\delta)} - 1\right| < \epsilon.
\end{align*}

Consequently, for $0 < \delta < \delta_2$, we have 
\begin{align}\label{step2}
    |V_{1+\delta}(\zeta) - V_{\infty} e^{-B e^{-\frac{2 \tau c}{\mu}\zeta}}| < \epsilon
\end{align}
for any $\zeta \in [-M, \infty)$.

\medskip
\noindent
\textbf{Step 3.} Now, choose $\delta_{0} = \min{(\delta_1, \delta_2)}>0$. Hence, for any $0<\delta < \delta_{0}$ and any $\zeta \in \mathbb{R}$ it follows from \eqref{step1} and \eqref{step2} that $$|V_{1+\delta}(\zeta)-V_{1}(\zeta)|<\epsilon.$$
Therefore, $V_{1+\delta}(\zeta)$ converges uniformly to $V_{1}(\zeta)$ in $\zeta \in \mathbb{R}$ as $\delta \to 0$.
\end{proof}

\begin{corollary}[Shapes of $U$ and $V$]\label{shape_mod1}
Let $U$ and $V$ be the solutions of the model \eqref{mod1}. 
\begin{enumerate}
    \item[\rm(i)]The function $U(\zeta)$ attains its maximum
    \[  U_{max} = \begin{cases} 
      2 c^{2} \tau k^{-1} \mu^{-1} V_{\infty} d^{-(\frac{d}{d-1})}  & \text{  at  } \zeta_0=\frac{\mu}{2\tau c}\ln(\frac{\frac{1}{2}C_{0} k c^{-2} \tau^{-1}\mu}{V_{\infty}^{-d+1}}) \text{  for  } d>1,\\
      2 c^{2} \tau k^{-1} \mu^{-1} V_{\infty}e^{-1} & \text{  at  } \zeta_0=\frac{\mu}{2\tau c}\ln(\frac{1}{2}C_{0} k c^{-2} \tau^{-1}\mu) \text{  for  } d=1.
   \end{cases}\]
    In fact, $U(\zeta)$ is an increasing function on $(-\infty , \zeta_{0})$ and decreasing on $(\zeta_0 , \infty)$.\\
    \item [\rm(ii)] The function $V(\zeta)$ is strictly increasing from $zero$ to constant $V_\infty$ as $\zeta$ varies from $-\infty$ to $\infty$.
\end{enumerate}
\end{corollary}


\subsection{The case of limited substrates}
If the  unavailability of the source of food plays a role in the depletion of the concentration of substrate then $k(v)$ is proportional to $v(x,t)$, i.e. $k(v) = k v$ with $k$ as a constant. Therefore the chemotactic model for the limited substrate is reduced to
\begin{subequations}\label{mod3}
\begin{align}
     \tau \frac{\partial u}{\partial t} +\beta \frac{\partial }{\partial x}\left(u \frac{\partial \ln v}{\partial x}\right) - \frac{\mu}{2}\frac{\partial^2 u}{\partial x^2}&=0,\label{ks1_b_sys}\\
     \frac{\partial v}{\partial t} + k u v\label{ks1_s_sys}&=0.
\end{align}
\end{subequations}

With Eq.~\eqref{chng var}, the system of equations \eqref{ks1_b_sys} and \eqref{ks1_s_sys} reduce to
\begin{align}
  \tau c U^{'} - \beta \left(U(\ln V)^{'}\right)^{'} +\frac{\mu}{2} U^{''} =0, \label{ks_red_b_sys} \\
  c V^{'} -k U V=0.\label{ks_red_s_sys}  
\end{align}
Note that, Eq.~\eqref{ks_red_s_sys} gives
\begin{align*}
    (\ln V)^{'} = \frac{k}{c}U, 
\text{       which implies      }    (\ln V)^{''}=\frac{k}{c}U^{'}.
\end{align*} 
Therefore from Eqs.~\eqref{ks_red_b_sys} and \eqref{ks_red_s_sys}, it follows that
\begin{equation}\label{main-ODE-2-model}
    U^{'} - C_{3} (U^2)^{'}+\frac{\mu}{2 \tau c} U^{''}=0.
\end{equation}
Here
\begin{equation*}\label{alpha}
    C_3 =\frac{2 \beta k}{c \mu}\end{equation*}
with dimension $[L]$.\par
Then integration of Eq.~\eqref{main-ODE-2-model} gives,
\begin{equation}\label{mod_3_U_int1}
    U - C_{3} U^2+\frac{\mu}{2 \tau c} U^{'}= constant.
\end{equation}
We are looking for solutions that satisfy condition~\eqref{bd_cond1}. And by the condition \eqref{bd_cond1}, the constant in Eq.~\eqref{mod_3_U_int1} is 0. \par
Then from above, it follows,
\begin{align}\label{1-st-order-ode-2}
  U^{'}&=C_{3}U(U- C_{4})
\end{align}
with
\begin{equation*}\label{K-def}
C_{4}=\frac{\tau c^2}{\beta k}.
\end{equation*}
Note that, $C_{3}>0$ and $C_4>0$.\par
\textbf{Case I:} $U(0) = C_{4}$. Then $U(\zeta) = C_{4}$ and $V(\zeta) = V(0) e^{\frac{kC_{4}}{c}\zeta}$ for all $\zeta \in \mathbb{R}$. These solutions do not satisfy the condition~\eqref{bd_cond1}.\par

\textbf{Case II:} $0<U(0) < C_4$. Then $0<U(\zeta)<C_{4}$ for all $\zeta \in \mathbb{R}$. 
Then partial decomposition of Eq.~\eqref{1-st-order-ode-2} gives, 
\begin{align*}\label{1-st-order-ode-2}
    \frac{U^{'}}{C_4 -U}+\frac{U^{'}}{U} & = -\frac{2 \tau c }{\mu}.
\end{align*}
And integration gives,
\begin{equation}\label{lim_neq_u_ln}
    \ln{\big|\frac{C_{4} -U}{U}\big|}= -\frac{2 \tau c}{\mu}\zeta + Constant.
\end{equation}
Therefore, from Eq.~\eqref{lim_neq_u_ln} we get
\begin{equation}\label{lim_neq_u_temp}
    U= C_{4} \big(1+C_{0}e^{\frac{2 \tau c}{\mu}\zeta}\big)^{-1}
\end{equation}
Hence $U(\zeta) \to C_4$ as $\zeta \to -\infty$ and $U(\zeta) \to 0$ as $\zeta \to \infty$.\par
Substituting Eq.~\eqref{lim_neq_u_temp} into Eq.~\eqref{ks_red_s_sys} and integrating, we get
\begin{align}\label{lim_neq_v}
 V & = C_6 \big(e^{-\frac{2 \tau c}{\mu}\zeta}+C_0)^{-\frac{\mu} {2\beta}},
\end{align}
where $C_6$ is the integrating constant.\par
With condition \eqref{bd_cond1}, it follows
\begin{align*}
 V = V_{\infty} \left(1+C_{7} e^{-\frac{2 \tau c}{\mu}\zeta} \right)^{-\frac{\mu} {2\beta}}
\end{align*}
with $C_{7} = C_{0}^{-1}$. Here, $V(\zeta) \to V_{\infty}$ as $\zeta \to \infty$. Therefore, the solutions satisfy the condition~\eqref{bd_cond1}. \par
\textbf{Case III:} $U(0) > C_4$. Then $U(\zeta)> C_{4}$ for all $\zeta \in (-\infty, \zeta_{max})$ where $\zeta_{max} \leq \infty$. Hence Eq.~\eqref{lim_neq_u_ln} gives \begin{align*}
     U= C_{4} \big(1-C_{0}e^{\frac{2 \tau c}{\mu}\zeta}\big)^{-1}.
\end{align*}
Here $\zeta_{max} = -\frac{\mu}{2 \tau c} \ln{C_5}$, which is finite and $U(\zeta) \to \infty$ as $\zeta \to \zeta_{max}$. Hence, the solutions do not satisfy the condition~\eqref{bd_cond1}.

\begin{theorem}\label{lim_food}
 For any $\tau, \beta, \mu, k, c>0$ the system~\eqref{mod3} has a traveling band of the form \eqref{chng var}. More precisely, $U(\zeta)$ and $V(\zeta)$ can be given by
\begin{subequations}
 \begin{align}
    U(\zeta)&= \frac{\tau c^2}{\beta k} \big(1+C_{0}e^{\frac{2 \tau c}{\mu}\zeta}\big)^{-1},\label{lim_neq_u}\\
    V(\zeta) &= V_{\infty} \left(1+ C_{0}^{-1} e^{-\frac{2 \tau c}{\mu}\zeta} \right)^{-\frac{\mu} {2\beta}}\label{lim_neq_v_cond},
\end{align}
\end{subequations}
where $V_{\infty}>0$ and $C_{0} > 1$. In fact, $U(\zeta)$ and $V(\zeta)$ in Eqs.~\eqref{lim_neq_u} and \eqref{lim_neq_v_cond} are unique solutions of Eqs.~\eqref{ks_red_b_sys} and \eqref{ks_red_s_sys} that satisfy condition \eqref{bd_cond1} and 
\begin{align}\label{mod3_c0}
    U(0) = \frac{\tau c^2 C_0}{\beta k}.
\end{align}
\end{theorem}

\begin{proof}
Suppose $C_{0} > 1$ and $U(0)$ satisfies Eq.~\eqref{mod3_c0}. Then we have $0<U(0) < C_4$, and the result follows case II above.
\end{proof}
One can see from Eqs.~\eqref{lim_neq_u} and \eqref{lim_neq_v_cond} that 
\begin{align}\label{lim_lim}
    \lim_{\zeta \xrightarrow{} -\infty}U(\zeta) = \frac{\tau c^2}{\beta k} \text{     and     } \lim_{\zeta \xrightarrow{} -\infty}V(\zeta) = 0.
\end{align}

\begin{theorem}
Suppose $\Vec{W} = (C_0, k, \tau, c, \beta, \mu, V_{\infty}) \in \mathbb{D} = (1, \infty) \times (0, \infty)^6$. Then $U(\zeta, \Vec{W})$ is continuous in $\Vec{W}$ uniformly in $\zeta \in \mathbb{R}$. More precisely, for any compact set $\Omega \subset \mathbb{D}$, there exists $K>0$ such that, for all $\Vec{W_1}, \Vec{W_2} \in \Omega$, 
\begin{align*}
    |U(\zeta, \Vec{W_1})- U(\zeta, \Vec{W_2})|&\leq K |\Vec{W_1}-\Vec{W_2}|,\\
    |V(\zeta, \Vec{W_1})- V(\zeta, \Vec{W_2})|&\leq K |\Vec{W_1}-\Vec{W_2}|
\end{align*}
for all $\zeta \in \mathbb{R}$.
\end{theorem}


\subsection{Discussion}\label{dis}
In this section, we provide analyses of obtained closed-form solutions of both models (\eqref{mod1} and \eqref{mod3}. We will discuss the traveling band phenomena in each case. All results are qualitative and we used the following listed parameter values that are adapted from published data \cite{Adleramino} and \cite{KELLER1971235} except $\tau$ for comparison. 

\begin{table}[h]
    \centering
    \caption{Parameter values}
    \begin{tabular}{c c c c}
    \hline
    Parameter & Description & Value & Units \\
    \hline
    $\tau$ & time interval of collision & 0.05-0.005 & hour\\
    $\mu$ & motility coefficient & 0.25 & cm$^2$/hour\\
    $c$ & band speed & 1.5 & cm/hour\\
    $\beta$ & chemotactic coefficient & 0.16-0.6 & cm$^2$/hour\\
    $d$ & $\frac{2 \beta}{\mu}$ & 0.3 - 10 & unit less\\
    $C_0$ & Integrating constant & 4 & unit less\\   \hline
    \end{tabular}
    \label{pv}
\end{table}

Graph of the solutions of the concentration of organism for the system \eqref{mod1} for different values of $\tau$ is given in Fig.~\ref{fig:unlim_eq_tau}. It depicts the dependence of the size of the traveling band on the time collision $\tau$. The band of bacteria gets wider as the value of $\tau$ gets smaller. 

\begin{figure}[!htbp]
    \centering
    \includegraphics[scale=0.5]{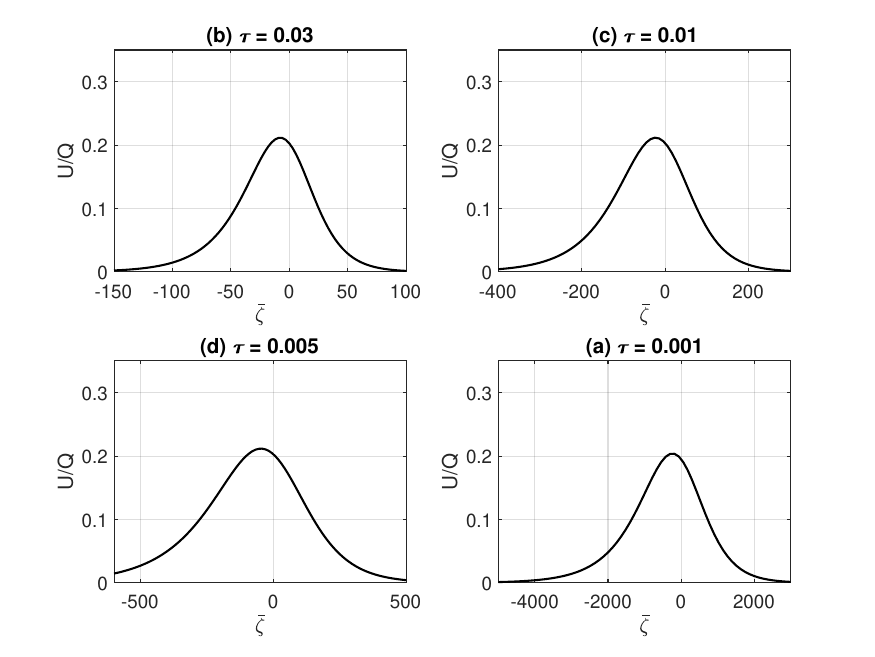}
    \caption{Concentration of bacteria $U(\zeta)$ divided by $Q =  C_0 V_{\infty}$ of model \eqref{mod1} for different values of $\tau$ with $d = 1.3$ against $\Bar{\zeta} = c \mu^{-1} \zeta$ showing the dependence of the size of bacterial band on the parameter $\tau$. The change in the range of the $\Bar{\zeta}$ axis is equivalent to the width of the band.}
    \label{fig:unlim_eq_tau}
\end{figure}

The family of the solutions of the model \eqref{mod1} converges uniformly to the solution $U_1$ and $V_1$. Fig.~\ref{fig:unlim_conv} shows that uniform convergence of (a) the concentration $U$ and (b) the concentration $V$. The dashed black line is the solution $U_1$ and $V_1$ in (a) and (b) respectively. 
Fig.~\ref{fig:unlim_conv}(a) shows the peaks of $U(\zeta)$ rise higher as the values of $d$ get smaller.
\begin{figure}[!htbp]
    \centering
    \includegraphics[scale=0.5]{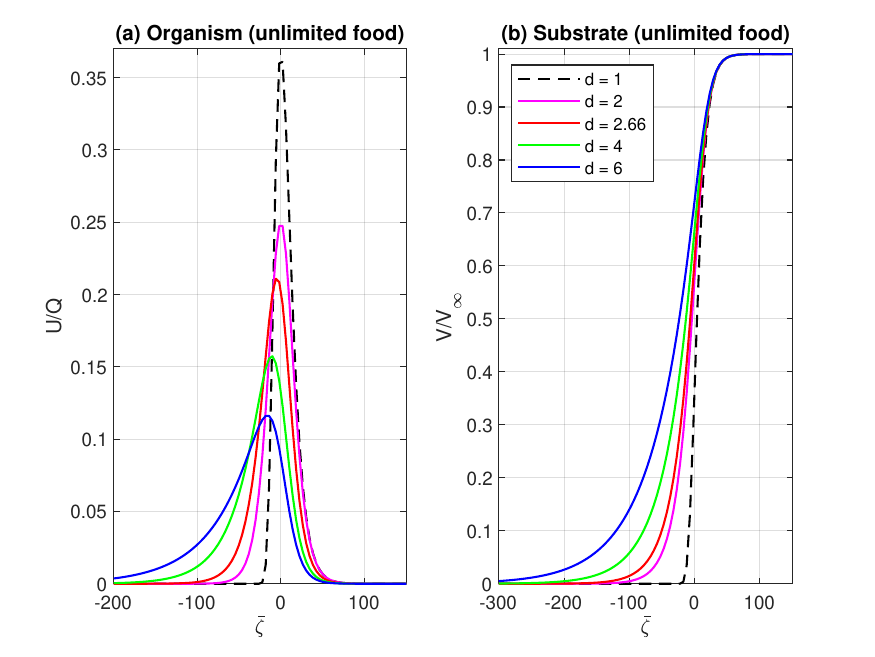}
    \caption{Showing the uniform convergence of (a) the concentration of bacteria $U(\zeta)$ divided by $Q =  C_0 V_{\infty}$ to the solution $U_1$ and (b) the concentration of the substrates $V(\zeta)$ divided by $V_{\infty}$ of model \eqref{mod1} to the solution $V_1$ as $d \to 1$.}
    \label{fig:unlim_conv}
\end{figure}

The concentration of organism and of the substrate of the model \eqref{mod3} in Figs.~\ref{fig:lim_neq} and \ref{fig:lim_neq_sub} for different values of $d$. Concentration of $U$ converge to 0 as $\zeta \to \infty$. For large negative values of $\zeta$, the concentration $U$ converges to a constant that gets reduced in size as $d$ gets larger. Therefore, in the presence of limited food, if we fix the location and look ahead for a long time, the concentration of the organism will converge to a smaller constant if the chemotaxis is bigger. And, in the case of substrates, the curve of $V$ gets flattered as $d$ increases. 
\begin{figure}[!htbp]
    \centering
    \includegraphics[scale=0.5]{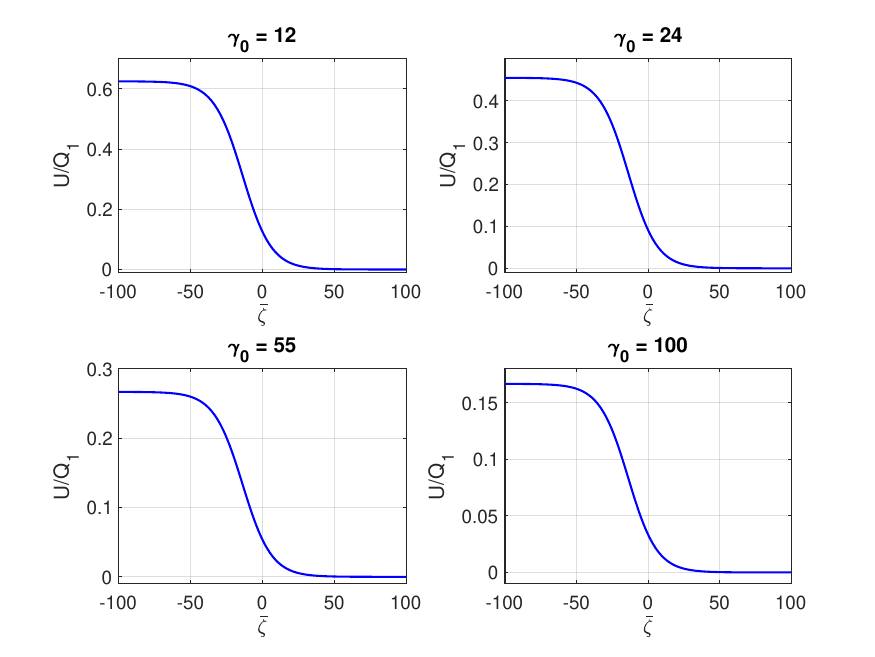}
    \caption{Concentration of bacteria $U(\zeta)$ divided by $Q_{1} =  2 \tau c^{2} k^{-1}\beta^{-1}$ of model \eqref{mod3} for different values of $d$ against $\zeta = c \mu^{-1} \zeta$.}
    \label{fig:lim_neq}
\end{figure}

\begin{figure}[!htbp]
    \centering
    \includegraphics[scale=0.5]{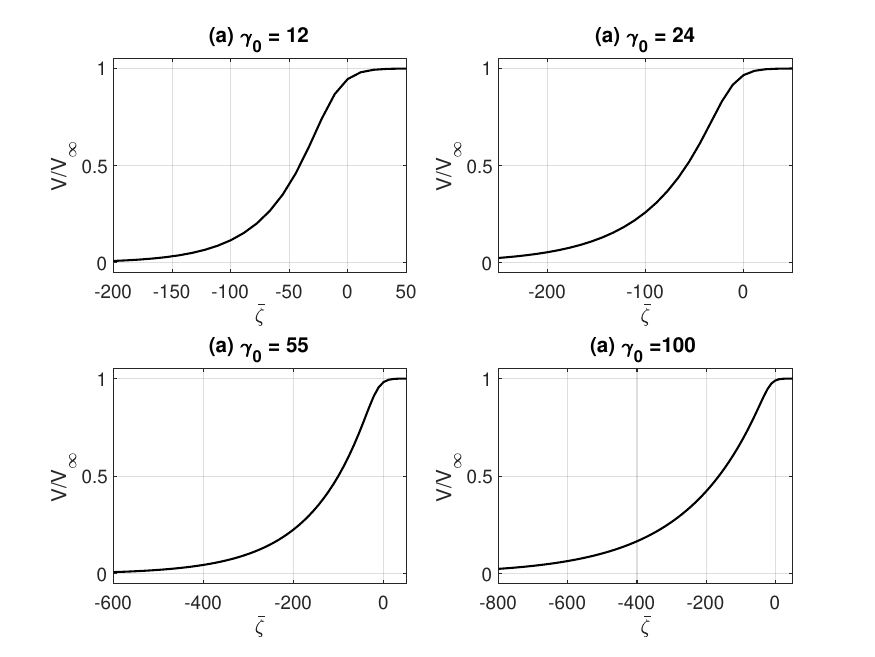}
    \caption{Same as in Fig. \ref{fig:lim_neq} for the concentration of the substrates $V(\zeta)$ divided by $V_{\infty}$ of model \eqref{mod3}.}
    \label{fig:lim_neq_sub}
\end{figure}
\section{Linear stability analysis of the model}\label{lstab}

In this section, we focused on the classic Keller-Segel type model based on the derivation in \eqref{Einstein}. 
Rewriting \eqref{main_mod} with $u(x,t)$ being the concentration of bacteria and $v(x,t)$ of the chemical substrate or attractant, we get
\begin{subequations}\label{slime_mod}
\begin{align}
    \tau \frac{\partial u}{\partial t} + \beta \frac{\partial}{\partial x}\left(u \frac{\partial \ln v}{\partial x}\right) -\frac{\mu}{2}\frac{\partial^2 u}{\partial x^2}&= 0,\label{slime_1}\\
    \tau \frac{\partial v}{\partial t} - \frac{D}{2} \frac{\partial^2 v}{\partial x^2} + H(u,v)&=0\label{slime_2}.
\end{align}
\end{subequations}

In Eqs.~\eqref{slime_1} and \eqref{slime_2}, $\beta$ is the chemoattractant coefficient, $\mu$, and $D$ are the diffusion coefficients of bacteria and chemical substrate respectively, $H(u,v)$ represents the net production of chemical substrate.

Our primary goal is to study the process of aggregation which indicates instability in the uniform configuration of bacteria and chemical substrate. We Assume a homogeneous population of cells throughout the system prior to aggregation. The uniform distribution for $u(x,t)$ and $v(x,t)$ are called to be an equilibrium state. Let's denote these equilibrium distributions by $u_0$ and $v_0$ accordingly. Here, $u_0>0$ and $v_0>0$ both are constants and $H(u_0,v_0)=0$. We are looking for solutions for Eqs.~\eqref{slime_mod} near equilibrium.

We begin by considering the domain $\Omega = (0,L), L>0$ and the small perturbations, 
\begin{align}\label{pert}
    u = u_0+\Bar{u}(x,t) \text{   and   } v = v_0 +\Bar{v}(x,t).
\end{align}

Then the corresponding homogeneous Dirichlet boundary conditions on the left boundary and Neumann boundary conditions on the right boundary are
\begin{align}
    \Bar{u}(0,t)= \Bar{v}(0,t)=0,\label{bd_d}\\
    \Bar{u}_x(L,t)= \Bar{v}_x(L,t)=0\label{bd_n}.
\end{align}

Assuming $\Bar{u}$, $\Bar{u}_x$, $\Bar{v}$, $\Bar{v}_x, \Bar{v}_{xx}$ to be small and using Eq.~\eqref{pert} from Eq.~\eqref{slime_1}, we get
\begin{align*}
    \tau \frac{\partial \Bar{u}}{\partial t}+ \beta \frac{u_0}{v_0}\frac{\partial^2 \Bar{v}}{\partial x^2} -\frac{\mu}{2}\frac{\partial^2 \Bar{u}}{\partial x^2}=0.
\end{align*}

Then, since $H(u_0,v_0)=0$ and ignoring higher-order terms, we have the following approximation
\begin{align}\label{Huv_comp}
     H(u,v) \approx \left(\frac{\partial H}{\partial u}(u_0,v_0)\right)\Bar{u} + \left(\frac{\partial H}{\partial v}(u_0,v_0)\right)\Bar{v}=-a\Bar{u}+d\Bar{v}.
\end{align}

Here, $a=-\frac{\partial H}{\partial u}(u_0,v_0)$ is the production of chemical substrate and $d=\frac{\partial H}{\partial v}(u_0,v_0)$ is the degradation of substrate.

Substituting Eq.~\eqref{Huv_comp} in Eq.~\eqref{slime_2}, we get equations for perturbation $\bar{u}(x,t)$ and $\bar{v}(x,t)$ as follows

\begin{align}
   \tau \frac{\partial \Bar{u}}{\partial t}+ \beta \frac{u_0}{v_0}\frac{\partial^2 \Bar{v}}{\partial x^2} -\frac{\mu}{2}\frac{\partial^2 \Bar{u}}{\partial x^2}=0.\label{pert_u} \\ 
   \tau \frac{\partial \Bar{v}}{\partial t} - \frac{D}{2} \frac{\partial^2 \Bar{v}}{\partial x^2} - a\Bar{u}+d\Bar{v}=0.\label{pert_v} 
\end{align}

Thanks to the above biological meaning of $a$ and $d$, we can naturally assume the following.
\begin{assumption}\label{sign}
Assume that $a>0$ and $d>0$.
\end{assumption}

To establish conditions for instability, we look for solutions in the form
\begin{align}\label{eigen_soln}
    \Bar{u}(x,t) = u^* e^{\sigma t} \Phi_{\lambda}(x) \text{   and   } \Bar{v}(x,t) = v^* e^{\sigma t} \Phi_{\lambda}(x).
\end{align}

Here, $u^*, v^*>0$ and $\Phi_{\lambda}(x)$ is a non-zero function that satisfies 
\begin{align}\label{ev}
    \Phi_{\lambda}''(x) = - \lambda^2 \Phi_{\lambda}(x) \text{   on   } \Omega, \ \ \Phi_{\lambda}(0)=0,  \text{  and  } \Phi_{\lambda}'(L)=0.
\end{align}

In other words, $\Phi_{\lambda}(x)$ is an eigenfunction of the eigenvalue problem \eqref{ev}. Explicitly, $\lambda = \lambda_n$, where
$$\lambda_n=\frac{1}{L}\left(2 \pi n+\frac{\pi}{2}\right) \text{     for    }n\in \{0\} \cup \mathbb{N},$$
and $\Phi_{\lambda}(x) = \Phi_{\lambda_n}(x) = \sin(\lambda_n x)$.

Let $\lambda = \lambda_n$ for a fixed number $n\in \{0\} \cup \mathbb{N}$. Then substituting Eq.~\eqref{eigen_soln} in Eq.~\eqref{pert_u} and \eqref{pert_v}, we obtain
\begin{subequations}\label{lin_sys}
    \begin{align}
        \big(\tau \sigma+ \frac{\mu
        }{2} \lambda^2 \big) u^* - \beta \frac{u_0}{v_0} \lambda^2 v^* & = 0,\\
        -a u^* + (\tau \sigma+ \frac{D}{2}\lambda^2 +d) v^* &=0.
    \end{align}
\end{subequations}

A non-trivial solution $(u^*, v^*)$ exists if and only if
\begin{align*}
\det\begin{pmatrix}
\tau \sigma+ \frac{\mu
}{2} \lambda^2 & - \beta \frac{u_0}{v_0} \lambda^2\\
-a & \tau \sigma+ \frac{D}{2}\lambda^2 +d \\
\end{pmatrix} =0.
\end{align*}
Then we get the following quadratic equation
\begin{align}\label{quad}
    \sigma^2 + b \sigma + c=0
\end{align}
with 
\begin{align}
    b &= \frac{1}{\tau} \left(\frac{D}{2}\lambda^2+\frac{\mu}{2}\lambda^2+d\right),\label{b}\\
    c &= \frac{\lambda^2}{\tau^2} \left(\frac{\mu D}{4}\lambda^2+ \frac{\mu d}{2} - a \beta \frac{u_0}{v_0}\right)\label{c_cond}.
\end{align}
The roots of Eq.~\eqref{quad} are
\begin{align}\label{eignv}
    \sigma_{i} = \frac{-b \mp \sqrt{b^2 - 4c}}{2} , \text{   for   } i = 1,2.
\end{align}

To analyze the stability, we need to investigate the sign of $\Re(\sigma_i)$ -- the real part of $\sigma_i$ -- in Eq.\eqref{eignv}. 
If $\Re(\sigma_i)>0$ for either $i=1$ or $i=2$, then perturbations would grow with time which corresponds to an unstable condition near the equilibrium. Recall, each of the parameters $\mu, \lambda, \tau, \beta, k, D, u_0$, $v_0$, $a$ and $d$ are positive. Therefore, the sign of $b$ is always positive and the sign of $c$ will determine the sign of $\Re(\sigma_i)$. Because $b>0$, we find that

\begin{center}
    \begin{tabular}[h]{c|c}
         \hline
         $c>0$ & $\Re(\sigma_i)<0$ for $i=1,2$\\
    $c<0$ & $\sigma_1<0, \sigma_2>0$\\
    $c=0$ & $\sigma_{1} = -b$, $\sigma_{2} = 0$\\
    \hline 
    \end{tabular}
    \label{sign_tab}
\end{center}

Thanks to the above table, the system is linearly unstable if $c<0$ which, by Eq.~\eqref{c_cond}, is equivalent to 
\begin{align}
    a> \frac{v_0}{2 u_0} \frac{\mu}{\beta} \left(\frac{D \lambda^2}{2}+d\right)= \frac{v_0}{2 u_0} \frac{\mu}{\beta} \left(\frac{D}{2L^2}\left(2 \pi n+\frac{\pi}{2}\right)^2+d\right)\label{ins_cond}.
\end{align}

Hence, we have the following theorem.
\begin{theorem}
Under Assumption \ref{sign}, if 
\begin{align}\label{n=0}
    a> \frac{v_0}{16 u_0} \frac{\mu}{\beta} \left(\frac{D\pi^2}{L^2}+8d\right),
\end{align}
then the trivial solution $(0,0)$ of the linearized system \eqref{pert_u} and \eqref{pert_v} is unstable. Consequently, the steady state $(u_0,v_0)$ of system \eqref{slime_mod} is linearly unstable.
\end{theorem}

\begin{proof}
Given $\epsilon>0$. Taking $n=0$, then $\lambda=\lambda_0$ and Eq.~\eqref{ins_cond} becomes \eqref{n=0}. Let $\sigma$ be the positive root of \eqref{quad}. Then there exist solutions $u^*$ and $v^*$ belonging to $(0, \epsilon)$ of the system \eqref{lin_sys}. Note that 
\begin{align*}
    \Phi_{\lambda}(x) = \Phi_{\lambda_0}(x) = \sin{\left(\frac{\pi x}{2L}\right)} \in (0,1) \text{     for all    } 0<x<L.
\end{align*} 

It follows that the solution $(\Bar{u}(x,t), \Bar{v}(x,t))$ given by \eqref{eigen_soln} of the linearized system \eqref{pert_u}, \eqref{pert_v} satisfies $\Bar{u}(x,0), \Bar{v}(x,0) \in (0, \epsilon)$ for all $x \in (0, L)$, and 
$$\lim_{t \to \infty} \Bar{u}(x,t) = \infty \text{     and     } \lim_{t \to \infty} \Bar{v}(x,t) = \infty \text{     for any     } x \in (0,L).$$

Therefore, the trivial solution $(0,0)$ of the linearized system \eqref{pert_u} and \eqref{pert_v} is unstable.
\end{proof}


The examination of the instability condition \eqref{ins_cond} tells us that instability occurs when the production of the chemical substrate passes a certain threshold. 
The production of the substrate must outweigh the local diffusion of bacteria and substrate. Also, as $L$ gets bigger there is more possibility for $a$ to overcome the threshold and make the onset of aggregation. Thus, in a large domain, less production of substrate instigates the process of aggregation. A large value of chemotactic factor or a slow degradation in substrate level tends to make the perturbations grow.       


\subsection{Linear stability via the energy method}\label{stab}

This section is dedicated to establishing the linear stability for the system \eqref{slime_mod} by using the linearized system \eqref{pert_u} and \eqref{pert_v}. For convenience, we re-denote $u = \Bar{u}$ and $v = \Bar{v}$ in Eqs.~\eqref{pert_u}, \eqref{pert_v} and conditions \eqref{bd_d}, \eqref{bd_n}. Therefore, we have the system 
\begin{subequations}\label{sys_stab}
    \begin{align}
    \tau \frac{\partial u}{\partial t}+ \frac{u_0}{v_0} \beta \frac{\partial^2 v}{\partial x^2} -\frac{\mu}{2}\frac{\partial^2 u}{\partial x^2}=0,\label{pert_u_stab}\\
    \tau \frac{\partial v}{\partial t} - \frac{D}{2} \frac{\partial^2 v}{\partial x^2} - au+dv=0,\label{pert_v_stab}
\end{align}
\end{subequations}
with the boundary conditions
\begin{align}
    u(0,t)= v(0,t)=0,\label{bc1}\\
    u_x(L,t)= v_x(L,t)=0\label{bc2}.
\end{align}

In studying the stability of the PDE system \eqref{sys_stab}, we will use the following well-known inequalities in Lemmas \ref{c_p} and \ref{cp} below. 
\begin{lemma}[Poincar\'e's inequality]\label{c_p}
Let $L>0$. If $u\in C^1([0,L])$ satisfies 
$$u(0)=0 \text{     or     } u(L)=0,$$ 
then 
\begin{align}\label{poin}
    \int_0^L \left(u'(x)\right)^2 dx \geq C_p \int_0^L u^2(x) dx,
\end{align}
where
\begin{equation}\label{cp-const}
 C_p=\frac{2}{L^2}.   
\end{equation} 
\end{lemma}

\begin{proof}
Although inequality is well known, we quickly present its proof for the sake of completeness. First, consider $u(0) = 0$. For $x \in [0,L]$, one has 
\begin{equation*}
    u^2(x) =\left(\int_0^x u'(\xi)d\xi\right)^2 \leq\left(\int_0^x|u'(\xi)|d\xi \right)^2.
\end{equation*}
Applying Holder's inequality gives
\begin{equation*}
    u^2(x) \leq x\cdot\left(\int_0^x |u'(\xi)|^2 d \xi\right)\leq x\cdot\left(\int_0^L |u'(\xi)|^2 d\xi\right).
\end{equation*}
Integrating in $x$ yields
\begin{equation*}
    \int_0^L u^2(x)dx\leq \frac{L^2}{2}\cdot\int_0^L |u'(x)|^2 dx,
\end{equation*}
which proves the inequality \eqref{cp-const}.

Next, consider $u(L) = 0$. For $x \in [0,L]$, one has 
\begin{equation*}
    u^2(x) =\left(-\int_x^L u'(\xi)d\xi\right)^2 \leq \left(\int_x^L|u'(\xi)|d\xi \right)^2.
\end{equation*}
Applying Holder's inequality gives
\begin{equation*}
    u^2(x) \leq (L-x) \cdot\left(\int_0^x |u'(\xi)|^2 d \xi\right)\leq (L-x)\cdot\left(\int_0^L |u'(\xi)|^2 d\xi\right).
\end{equation*}
Integrating in $x$ yields
\begin{equation*}
    \int_0^L u^2(x)dx\leq \frac{L^2}{2}\cdot\int_0^L |u'(x)|^2 dx,
\end{equation*}
which proves the inequality \eqref{cp-const}.
\end{proof}

\begin{lemma}\label{cp}
Let $L>0$. If $u\in C^1([0,L])$ satisfies 
$$u(0)=0 \text{     or     } u(L)=0,$$ 
then 
\begin{align}\label{poinl}
    \sup_{x \in [0,L]}\left|{u(x)}\right| \le \sqrt{L} \left(\int_0^L \left|u'(x)\right|^2 dx\right)^{\frac{1}{2}}.
\end{align}

\end{lemma}

\begin{proof}
Again, an elementary proof is given here for the sake of completeness. Let $x \in [0,L]$. If $u(0) = 0$, then by the Fundamental Theorem of Calculus and H\"older's inequality, one has 
\begin{equation*}
    \left|u(x)\right| =\left|\int_0^x u'(\xi)d\xi\right| \leq\int_0^L|u'(\xi)|d\xi \le \sqrt{L} \left(\int_0^L \left|u'(\xi)\right|^2 d\xi\right)^{\frac{1}{2}}.
\end{equation*}

If $u(L) = 0$, then one similarly has
\begin{equation*}
    \left|u(x)\right| =\left|-\int_x^L u'(\xi)d\xi\right| \leq\int_0^L|u'(\xi)|d\xi \le \sqrt{L} \left(\int_0^L \left|u'(\xi)\right|^2 d\xi\right)^{\frac{1}{2}}.
\end{equation*}

Therefore, we obtain the inequality \eqref{poinl}.
\end{proof}

First, we have the following stability result for the $L^2$-norm.

\begin{theorem}
If 
\begin{align}
    \frac{2 u_0}{v_0} \beta \le \mu \text{     and     } \frac{2 u_0}{v_0} \beta \le D, \label{mu_D_stab}\\
    4 a \le C_p \mu \text{     and     } 4 a \le C_pD + d \label{mu_D2_stab},
\end{align}
then there is a number $A>0$ such that 
\begin{align}\label{stab_eq}
    \int_0^L (u^2 (x,t)+ v^2 (x,t)) dx \le e^{-At}\int_0^L (u^2 (x,0) + v^2 (x,0)) dx 
\end{align}
for all $t\ge0$.

Consequently, the steady state $(u_0,v_0)$ of system \eqref{slime_mod} is linearly stable with respect to the $L^2$-norm over the interval $[0,L]$.
\end{theorem}

\begin{proof}
Multiplying Eq.~\eqref{pert_u_stab} by $u$ and Eq.~\eqref{pert_v_stab} by $v$ and integrating in $x$ over $(0,L)$, we have
\begin{align*}
    \frac{\tau}{2} \frac{d }{d t}\int_0^L u^2 dx + \frac{u_0}{v_0} \beta \int_0^L u \frac{\partial^2 v}{\partial x^2} dx -\frac{\mu}{2} \int_0^L u \frac{\partial^2 u}{\partial x^2} dx=0,\\
    \frac{\tau}{2} \frac{d }{d t} \int_0^L v^2 dx- \frac{D}{2} \int_0^L v \frac{\partial^2 v}{\partial x^2} dx- a \int_0^L v u dx+d \int_0^L v^2 dx=0.
\end{align*}

Integration by parts and adding both equations yields
\begin{align}
    &\frac{\tau}{2} \frac{d }{d t}\left(\int_0^L u^2 dx+ \int_0^L v^2 dx\right)+ \frac{\mu}{4} \int_0^L \left(\frac{\partial u}{\partial x}\right)^2 dx - \frac{u_0}{v_0} \beta \int_0^L \frac{\partial u}{\partial x} \frac{\partial v}{\partial x} dx \nonumber \\
    & \quad + \frac{D}{4} \int_0^L \left(\frac{\partial v}{\partial x}\right)^2 dx  
    +\frac{\mu}{4} \int_0^L \left(\frac{\partial u}{\partial x}\right)^2 dx+ \frac{D}{4} \int_0^L \left(\frac{\partial v}{\partial x}\right)^2 dx \nonumber \\
    & \quad - a \int_0^L v u dx + d \int_0^L v^2 dx=0\label{add}.
\end{align}

By the condition \eqref{mu_D_stab} and H\"older's inequality
\begin{align*}
    \frac{\mu}{4} \int_0^L \left(\frac{\partial u}{\partial x}\right)^2 dx - \frac{u_0}{v_0} \beta \int_0^L \frac{\partial u}{\partial x} \frac{\partial v}{\partial x} dx + \frac{D}{4} \int_0^L \left(\frac{\partial v}{\partial x}\right)^2 dx \nonumber \\
    \quad \ge \frac{1}{2} \frac{u_0}{v_0} \beta \left(\sqrt{\int_0^L \left(\frac{\partial u}{\partial x}\right)^2 dx} - \sqrt{\int_0^L \left(\frac{\partial v}{\partial x}\right)^2 dx}\right)^2\ge 0.
\end{align*}

Therefore, we obtain
\begin{align*}
    &\frac{\tau}{2} \frac{d }{d t}\left(\int_0^L u^2 dx+ \int_0^L v^2 dx\right)
    +\frac{\mu}{4} \int_0^L \left(\frac{\partial u}{\partial x}\right)^2 dx+ \frac{D}{4} \int_0^L \left(\frac{\partial v}{\partial x}\right)^2 dx \nonumber \\
    & \quad -a \int_0^L v u dx+ d \int_0^L v^2 dx\le0.
\end{align*}

Applying Poincar\'e's inequality \eqref{poin}, we get
\begin{align*}
    &\frac{\tau}{2} \frac{d }{d t}\left(\int_0^L u^2 dx+ \int_0^L v^2 dx\right)+ \frac{C_p\mu}{8} \int_0^L u^2 dx+ \left(\frac{C_p D}{8} +\frac{7 d}{8}\right) \int_0^L v^2 dx \nonumber \\
    & \quad +\frac{C_p\mu}{8} \int_0^L u^2 dx- a \int_0^L v u dx + \left(\frac{C_p D}{8} + \frac{d}{8}\right) \int_0^L v^2 dx  \le 0.
\end{align*}

By the condition \eqref{mu_D2_stab} and H\"older's inequality, we find
\begin{align*}
    \frac{C_p\mu}{8} \int_0^L u^2 dx- a \int_0^L v u dx+ \left(\frac{C_p D}{8} + \frac{d}{8}\right) \int_0^L v^2 dx \nonumber \\
    \quad \ge \frac{1}{2} a \left(\sqrt{\int_0^L u^2 dx} - \sqrt{\int_0^L v^2 dx}\right)^2\ge 0.
\end{align*}

Hence, it follows that
\begin{align}
    &\frac{\tau}{2} \frac{d }{d t}\left(\int_0^L u^2 dx+ \int_0^L v^2 dx\right)+ \frac{C_p\mu}{8} \int_0^L u^2 dx+ \left(\frac{C_p D}{8} +\frac{7 d}{8}\right) \int_0^L v^2 dx \le 0\label{energy}.
\end{align}

By denoting $C = \min\{\frac{C_p\mu}{4},\frac{C_p D}{4} +\frac{7 d}{4}\}$, we have from Eq.~\eqref{energy}  
\begin{align*}
    \tau \frac{d }{d t}I(t) + C I(t) &\le 0,
\end{align*}
with $I(t) = \int_0^L u^2 (x,t) dx+ \int_0^L v^2 (x,t) dx$.

By Gr\"onwall's inequality, we find
\begin{align*}
    I(t) \le I(0) e^{-\frac{C}{\tau}t}
\end{align*}
for all $t\ge 0$. Therefore, we obtain \eqref{stab_eq}.
\end{proof}

Next, we have the stability for the $H^1$-norm and $L^\infty$-norm.
\begin{theorem}
Assume \eqref{mu_D_stab} and \eqref{mu_D2_stab}. Then 
\begin{align}\label{stab_eq_2}
    \int_0^L (u_x^2 (x,t)+ v_x^2 (x,t)) dx \le M_0^2 e^{-Bt} \text{   
  for all $t\ge0$,}
\end{align}
where $B= a/\tau$ and
\begin{align*}
    M_0=\int_0^L (u_x^2 (x,0) + v_x^2 (x,0)) dx.
\end{align*}
Consequently, one has, for all $t \ge 0$,
\begin{align}\label{stab_eq_3}
     \sup_{x \in [0,L]}{\left|u (x,t)\right|}+ \sup_{x \in [0,L]}{\left|v (x,t)\right|}  \le 2 \sqrt{L} M_0 e^{-\frac{Bt}{2}} .
\end{align}

Therefore, the steady state $(u_0,v_0)$ of system \eqref{slime_mod} is linearly stable with respect to the $L^\infty$-norm over the interval $[0,L]$.
\end{theorem}

\begin{proof}
Multiplying Eq.~\eqref{pert_u_stab} by $-u_{xx}$ and Eq.~\eqref{pert_v_stab} by $-v_{xx}$ and integrating in $x$ over $(0,L)$, with the use of boundary conditions \eqref{bc1} and \eqref{bc2}, we have
\begin{align}
    \frac{\tau}{2} \frac{d }{d t}\int_0^L u_x^2 dx+\frac{\mu}{2} \int_0^L u_{xx}^2 dx - \frac{u_0}{v_0} \beta \int_0^L u_{xx} v_{xx} dx =0,\label{u_x1}\\
    \frac{\tau}{2} \frac{d }{d t} \int_0^L v_x^2 dx+ \frac{D}{2} \int_0^L v_{xx}^2 dx- a \int_0^L u_x v_x dx+d \int_0^L v_x^2 dx=0\label{v_x1}.
\end{align}

By the Cauchy inequality, 
\begin{align}\label{cauchy}
 \left|u_{xx} v_{xx}\right| \le \frac{1}{2} \left(u_{xx}^2+ v_{xx}^2\right),\quad \left|u_{x} v_{x}\right| \le \frac{1}{2} \left(u_{x}^2+ v_{x}^2\right).
\end{align}

Summing up Eqs.~\eqref{u_x1} and \eqref{v_x1} and using \eqref{cauchy}, we derive
\begin{align}
    &\tau \frac{d }{d t}\left(\int_0^L u_x^2 dx+ \int_0^L v_x^2 dx\right)+ \left(\mu- \frac{u_0}{v_0} \beta\right) \int_0^L u_{xx}^2 dx + \left(D- \frac{u_0}{v_0} \beta\right) \int_0^L v_{xx}^2 dx \nonumber \\
    & \quad - a \int_0^L u_x^2 dx + (2d-a) \int_0^L v_x^2 dx\le 0.
\end{align}

Applying the Poincar\'e's inequality to $u_x$ and $v_x$ with the boundary conditions $u_x(L,t)= v_x(L,t)=0$ yields
\begin{align*}
    \int_0^L u_{xx}^2(x,t) dx \ge C_p \int_0^L u_x^2(x,t) dx, \quad \int_0^L v_{xx}^2(x,t) dx \ge C_p \int_0^L v_x^2(x,t) dx.
\end{align*}

Note also 
\begin{align*}
    \mu- \frac{u_0}{v_0} \beta>0, \quad D- \frac{u_0}{v_0} \beta>0.
\end{align*}

It follows that
\begin{align*}
   &\tau \frac{d }{d t}\left(\int_0^L u_x^2 dx+ \int_0^L v_x^2 dx\right)+ \left(\mu C_p- \frac{u_0 C_p}{v_0} \beta-a \right) \int_0^L u_{x}^2 dx \\
   & \quad + \left(DC_p +2d- \frac{u_0 C_p}{v_0} \beta -a\right) \int_0^L v_{x}^2 dx \le 0.
\end{align*}

We have 
\begin{align*}
    &\mu C_p- \frac{u_0 C_p}{v_0} \beta-a \ge \mu C_p- \frac{1}{2}\mu C_p- a = \frac{1}{2}\mu C_p- a \ge 2a-a = a,\\
    &DC_p +2d- \frac{u_0 C_p}{v_0} \beta -a \ge DC_p +2d- \frac{1}{2} D C_p -a \nonumber \\
    & \quad = \frac{1}{2}(D C_p+ d)+ \frac{3}{2}d -a \ge 2a -a =a.
\end{align*}

Hence, we obtain
\begin{align*}
   &\tau \frac{d }{d t}\int_0^L \left(u_{x}^2 + v_{x}^2\right) dx+ a \int_0^L \left(u_{x}^2 + v_{x}^2\right) dx \le 0.
\end{align*}

By Gr\"onwall's inequality, we find
\begin{align*}
    I(t) \le I(0) e^{-\frac{a}{\tau}t}, \text{     for all $t\ge 0$ }
\end{align*}
with $I(t) = \int_0^L \left(u_x^2 (x,t)+ v_x^2 (x,t)\right) dx$. Therefore, we obtain \eqref{stab_eq_2}.

Then the estimate \eqref{stab_eq_3} follows by the use of inequality \eqref{poinl}. 
\end{proof}

Using the Lemma \ref{c_p}, the stability conditions \eqref{mu_D_stab} and \eqref{mu_D2_stab} can be rewritten as
\begin{subequations}\label{stab_cond}
\begin{align}
    \frac{v_0}{2 u_0} \frac{\mu}{\beta}\ge 1 \text{     and     }  \frac{v_0}{2 u_0} \frac{D}{\beta} \ge 1,\label{stab1}\\
    \frac{\mu}{2 a L^2} \ge 1 \text{     and     } \frac{D}{2 a L^2} + \frac{d}{2 a} \ge 1.\label{stab2}
\end{align}
\end{subequations}

The first condition \eqref{stab1} indicates that $\mu$ and $D$ must dominate the attraction of bacteria towards the attractor. It makes sense since amoeba gets attracted to the high concentration of substrate. As a large $D$ smoothing out the high concentration, it results in a small attraction, $\beta$, towards the substrate. Also, a large $\mu$ opposes the accumulation of bacteria and heads toward the direction of the presence of a chemical substrate. 

On the other hand, the second condition \eqref{stab2} illustrates that, for a stable condition, $\mu$ and $D$ must be larger than the production rate $a$. Also, a rapid degradation in substrate level, $d$, is needed to outweigh the production $a$. And lastly, the size of the domain, $L$, should be small enough so that diffusion of bacteria and substrate can spread over the whole space. 


\section{Conclusion} In this chapter, we have shown a way of incorporating Einstein's Brownian motion model in deducing a Keller-Segel-type chemotactic system. The systems with limited and unlimited substrates both have traveling band solutions. The linear stability analysis reveals that a small value of $D$ or/and $\mu$ compared to the substrate production instigates aggregation or instability of the system. On the other hand, the chemotactic attraction of the cell is expected to be higher for stability. Degradation in substrate level also brings stability. Finally, a bigger length of domain supports instability.


\section*{Acknowledgments}
We thank Prof.\ Dr.\ Luan~Hoang for his contributions in to writing out theorems 2 and 3 and stimulating discussion.


\nocite{*}
\bibliographystyle{abbrv}
\bibliography{references}

\end{document}